\documentclass[a4paper,12pt]{article}
\usepackage[hiresbb]{graphicx}
\usepackage{amsmath,amsthm,amssymb}
\usepackage{natbib,color,bm}
\numberwithin{equation}{section}
\usepackage{setspace} 
\usepackage{diagbox}
\usepackage[top=30truemm,bottom=30truemm,left=25truemm,right=25truemm]{geometry}

\usepackage{algorithmicx}%
\usepackage{algpseudocode}%
\usepackage{listings}%
\usepackage{anyfontsize}%
\usepackage{diagbox}
\usepackage{booktabs}
\usepackage{authblk}
\usepackage{multirow}
\usepackage{caption}
\usepackage{hyperref}
\hypersetup{%
  colorlinks=true,
  linkcolor=blue,  
  citecolor=black,   
  urlcolor=blue     
}

\makeatletter
\AtBeginDocument{%
  \@ifundefined{NAT@hyper@}{}{%
    \let\NAT@orig@hyper@\NAT@hyper@
    \renewcommand*{\NAT@hyper@}[1]{%
      \begingroup
        \let\NAT@orig@date\NAT@date
        \def\NAT@date{\textcolor{blue}{\NAT@orig@date}}%
        \NAT@orig@hyper@{#1}%
      \endgroup
    }%
  }%
}
\makeatother
 
\theoremstyle{plain} 
\newtheorem{thm}{Theorem}[section]%
\newtheorem{prop}[thm]{Proposition}%
\newtheorem{lem}[thm]{Lemma}%

\theoremstyle{definition}

\newtheorem{rem}{Remark}%

\allowdisplaybreaks

\numberwithin{equation}{section}

\title{\Large\textbf{Third-order Halley-type iterative method with positive and bounded correction function}}
\author[1,2]{Hirai Mukasa}

\affil[1]{Department of Creative Design Engineering, National Institute of Technology, Kagoshima College}
\affil[2]{Graduate School of Mathematics, Kyushu University}

\date{}

\begin{document}
\maketitle

\begin{abstract}
In this paper, we propose a new one-point iterative method for finding simple roots of nonlinear equations by modifying Halley’s method. 
Its correction function is positive and bounded on $\mathbb{R}$ and matches the local expansion of Halley’s correction function near a simple root.
The method avoids the singularity and sign reversal of Halley’s correction function while retaining at least third-order local convergence.
We also establish sufficient conditions for global convergence.
Numerical experiments show that the proposed method achieves high convergence success rates and favorable iteration counts over a wide range of initial guesses.
Further experiments near the singularity of Halley’s correction function demonstrate robust convergence behavior of the proposed method.
An application to the van der Waals equation illustrates the effectiveness of the proposed method.
\end{abstract}

\maketitle

\section{Introduction}\label{sec:1}	
Solving nonlinear equations $f(x)=0$ is a classical problem in numerical analysis with applications in a wide range of scientific and engineering fields \citep{RefW96,RefLZ04,RefD05}.
Since closed-form solutions are generally unavailable, iterative methods are widely used to compute approximate solutions.
Among these methods, Newton’s method is one of the most fundamental and widely used.
It is given by
\begin{align*}
x_{n+1} = x_n-\dfrac{f(x_n)}{f'(x_n)}, \quad n = 0,1,\ldots,
\end{align*}
which converges quadratically in a neighborhood of a simple root $\alpha$ \citep{RefT64, RefO73}.
However, the method may fail to converge when the initial guess is far from the root \citep{RefY00}.

Halley’s method is a one-point iterative method with a higher local order of convergence than Newton’s method.
The iteration is given by
\begin{align*}
x_{n+1} = x_n - \dfrac{2f(x_n)f'(x_n)} {2f'(x_n)^2-f(x_n)f''(x_n)},
\end{align*}
which converges cubically to $\alpha$ when the initial guess is sufficiently close to the root.
Its geometric properties and convergence behavior have also been extensively studied \citep{RefH91,RefM97}.
Other third-order iterative methods, including the Chebyshev, super-Halley, Fang, Ostrowski, and Noor methods, have also been proposed and analyzed from the viewpoints of local and global convergence \citep{RefHP77,RefGH01,RefF06,RefNN07,RefNS13,RefB20}.

\cite{RefC15} provided the following unified representation of a broad class of third-order one-point iterative methods:
\begin{align}\label{GE}
x_{n+1} = x_n-C(L_f(x_n))\dfrac{f(x_n)}{f'(x_n)},
\end{align}
where
\begin{align*}
L_f(x) = \dfrac{f(x)f''(x)}{f'(x)^2}.
\end{align*}
For Halley’s method, the correction function in \eqref{GE} is given by
\begin{align}\label{halley}
C_H(L_f(x)) = \dfrac{2}{2-L_f(x)}.
\end{align}
Although Halley’s method has excellent local convergence properties, its correction function has undesirable global features.
In particular, $C_H(L_f(x))$ has a pole at $L_f(x)=2$, where
\begin{align*}
\lim_{L_f(x)\rightarrow2^-}C_H(L_f(x))= \infty, \quad \lim_{L_f(x)\rightarrow2^+}C_H(L_f(x))=-\infty.
\end{align*}
Moreover, if $L_f(x)>2$, then $C_H(L_f(x))<0$, so the iteration proceeds in the direction opposite to the Newton step.
Thus, an iterate entering a region where $L_f(x)$ is close to or greater than two may undergo an excessively large correction or a reversal of the Newton direction.
These features do not affect the cubic local convergence of Halley’s method, since $L_f(x)\to0$ as $x\to\alpha$, but they can adversely affect its behavior when the iterate is far from the root.
In particular, \cite{RefM23} proposed the Extended Newton method with a positive correction function on $\mathbb{R}$ to address the undesirable behavior of $C_H(L_f(x))$ for $L_f(x)>2$. 
A method based on Pad\'e approximation was also considered to limit excessively large corrections.
However, both correction functions remain unbounded on $\mathbb{R}$.

The properties of the correction functions also vary among other third-order iterative methods.
For example, the correction functions of the Chebyshev and Noor methods become negative as $L_f(x)\to-\infty$ and are unbounded, whereas the super-Halley correction function has a pole on $\mathbb{R}$ \citep{RefO73,RefGH01,RefNN07}. 
The Fang and Ostrowski correction functions involve square roots and are therefore not real-valued on all of $\mathbb{R}$, although the former is positive and bounded on its real-valued domain \citep{RefO73,RefF06}.
Table~\ref{table1} summarizes relevant properties of correction functions for these one-point iterative methods.
Here, ``Real-valued on $\mathbb{R}$'' means that $C(L_f(x))\in\mathbb{R}$ for all $L_f(x)\in\mathbb{R}$, whereas ``Positive'' and ``Bounded'' refer to positivity and boundedness, respectively, on the real-valued domain of the correction function.
``Halley matching'' means that
\begin{align}\label{HM}
C(L_f(x)) = C_H(L_f(x))+O(L_f(x)^4) \quad (L_f(x) \rightarrow 0).
\end{align}
\begin{table}[t]
\centering
\caption{Properties of correction functions for third-order one-point iterative methods.}
\label{table1}
\renewcommand{\arraystretch}{1.15}
\begin{tabular}{lcccc}
\toprule
Method
& Real-valued on $\mathbb{R}$
& Positive
& Bounded
& Halley matching
\\
\midrule
Halley
& $\times$
& $\times$
& $\times$
& $\checkmark$
\\
Chebyshev
& $\checkmark$
& $\times$
& $\times$
& $\times$
\\
super-Halley
& $\times$
& $\times$
& $\times$
& $\times$
\\
Fang
& $\times$
& $\checkmark$
& $\checkmark$
& $\times$
\\
Ostrowski
& $\times$
& $\checkmark$
& $\times$
& $\times$
\\
Noor
& $\checkmark$
& $\times$
& $\times$
& $\times$
\\
Extended Newton
& $\checkmark$
& $\checkmark$
& $\times$
& $\times$
\\
Pad\'e-type
& $\checkmark$
& $\checkmark$
& $\times$
& $\times$
\\
Proposed
& $\checkmark$
& $\checkmark$
& $\checkmark$
& $\checkmark$
\\
\bottomrule
\end{tabular}
\end{table}
As shown in Table~\ref{table1}, none of the methods considered here satisfies all four properties.
In this study, we construct a correction function that is positive and bounded on $\mathbb{R}$ and whose local expansion matches that of Halley’s correction function near the root, and use it to define a new one-point iterative method.
The higher-order local matching near $L_f(x)=0$ preserves the cubic convergence behavior of Halley’s method, while the global shape of the proposed correction function is designed to improve robustness for initial guesses far from the root.
We also establish sufficient conditions for the global convergence of the proposed method.

The rest of the paper is organized as follows.
In Section~\ref{sec:2}, we construct the proposed correction function and derive the corresponding iterative method.
Section~\ref{sec:3} establishes the fundamental properties of the correction function and investigates the local and global convergence of the method.
Section~\ref{sec:4} presents numerical comparisons with existing third-order iterative methods and an application to the van der Waals equation.
Finally, Section~\ref{sec:5} gives the concluding remarks.

\section{Proposed method}\label{sec:2}	
In this section, we derive a correction function that is positive and bounded on $\mathbb{R}$  and preserves the local behavior of Halley’s correction function near the root.
Consider the third-order one-point iterative method \eqref{GE} introduced in Section \ref{sec:1}.
Halley’s correction function \eqref{halley} can be rewritten as
\begin{align*}
C_H(L_f(x)) = \dfrac{2}{1 + (1 - L_f(x))}.
\end{align*}
Replacing $1-L_f(x)$ by a function $B(L_f(x))$, we define
\begin{align*}
C_*(L_f(x)) = \dfrac{2}{1 + B(L_f(x))}.
\end{align*}
We require $B(L_f(x))$ to satisfy the following three conditions:
\begin{align}
& \textup{(i)} \ B(L_f(x)) > 0 \quad (L_f(x) \in \mathbb{R}). \label{D1}\\
& \textup{(ii)} \ B(L_f(x)) = 1 -L_f(x) + O(L_f(x)^4) \quad (L_f(x) \rightarrow 0). \label{D2}\\
& \textup{(iii)} \ \dfrac{B(L_f(x))}{1 - L_f(x)} \rightarrow 1 \quad (L_f(x) \rightarrow -\infty). \label{D3}
\end{align}
Condition \eqref{D1} guarantees that the proposed correction function is positive and bounded on $\mathbb{R}$.
Indeed, $0 < C_*(L_f(x)) < 2$ for all $L_f(x)\in\mathbb{R}$.
Condition \eqref{D2} yields the local matching property \eqref{HM}, since $L_f(x)\to0$ as $x\to\alpha$.
Condition \eqref{D3} yields the asymptotic relation
\begin{align*}
\dfrac{C_*(L_f(x))}{C_H(L_f(x))} \rightarrow 1 \quad (L_f(x)\to-\infty).
\end{align*}

To satisfy \eqref{D1}, we first seek a function that is positive on $\mathbb{R}$ and has the same constant and linear terms as $1-L_f(x)$. Consider the one-parameter family
\begin{align*}
R_{\mu}(L_f(x)) = \sqrt{1+\mu L_f(x)^2}-L_f(x), \quad \mu > 0.
\end{align*}
The condition $R_{\mu}(L_f(x)) > 0$ for all $L_f(x)\in\mathbb{R}$ holds if and only if $\mu \geq 1$.
By Taylor series expansion about $L_f(x) = 0$, we have 
\begin{align}\label{taylor}
R_{\mu}(L_f(x)) = 1-L_f(x) +\dfrac{\mu}{2}L_f(x)^2 -\dfrac{\mu^2}{8}L_f(x)^4 + O(L_f(x)^6).
\end{align}
Among $\mu\geq1$, the choice $\mu=1$ minimizes the discrepancy in the quadratic term from $1-L_f(x)$.
Therefore, we set
\begin{align*}
R(L_f(x)) = \sqrt{1 + L_f(x)^2} - L_f(x).
\end{align*}
However, the quadratic term in \eqref{taylor} does not coincide with that of $1-L_f(x)$.
Thus, we introduce a rational function $A(L_f(x))$ and construct $B(L_f(x)) = A(L_f(x))R(L_f(x))$ so that \eqref{D2} and \eqref{D3} are satisfied.
Suppose that
\begin{align*}
A(L_f(x)) = 1 + a_1 L_f(x) + a_2 L_f(x)^2 + a_3 L_f(x)^3 + O(L_f(x)^4).
\end{align*}
Using the Taylor expansion of $R(L_f(x))$, we obtain
\begin{align*}
B(L_f(x)) &= 1 + (-1 + a_1)L_f(x) + \left(\dfrac{1}{2} - a_1 + a_2\right)L_f(x)^2\\
& \quad + \left(\dfrac{1}{2}a_1 - a_2 + a_3\right)L_f(x)^3 + O(L_f(x)^4).
\end{align*}
Comparing the coefficients with \eqref{D2} gives
\begin{align*}
\left\{ \,
 \begin{aligned}
& -1 + a_1 = -1\\
& \dfrac{1}{2} - a_1 + a_2 = 0\\
& \dfrac{1}{2}a_1 - a_2 + a_3 = 0
\end{aligned}
    \right.
\quad \Leftrightarrow \quad 
\left\{ \,
 \begin{aligned}
& a_1 = 0\\
& a_2 = -\dfrac{1}{2}\\
& a_3 = -\dfrac{1}{2} 
\end{aligned}
    \right..
\end{align*}
Thus,
\begin{align}\label{eq:Alocal}
A(L_f(x)) = 1 - \dfrac{1}{2} L_f(x)^2 - \dfrac{1}{2} L_f(x)^3 + O(L_f(x)^4).
\end{align}
Here, we consider the following rational function of the lowest degree that satisfies \eqref{eq:Alocal} and is positive on $\mathbb{R}$:
\begin{align}\label{A}
A(L_f(x)) = \dfrac{1 + p_1 L_f(x) + p_2 L_f(x)^2}{1 + q_1 L_f(x) + q_2 L_f(x)^2}.
\end{align}
Taylor series expansion about $L_f(x) = 0$ yields
\begin{align*}
A(L_f(x))
&= 1+(p_1 - q_1)L_f(x) + (p_2 - p_1 q_1 + q_1^2 - q_2)L_f(x)^2\\
&\quad + \left[2q_1 q_2 - q_1^3 + p_1(q_1^2 - q_2)
- p_2 q_1\right]L_f(x)^3 + O(L_f(x)^4).
\end{align*}
Comparing the coefficients with those in \eqref{eq:Alocal}, we have
\begin{align*}
\left\{
\begin{aligned}
& p_1 - q_1=0\\
& p_2 - p_1 q_1 + q_1^2 - q_2 = -\dfrac{1}{2}\\
& 2q_1 q_2 - q_1^3 + p_1(q_1^2-q_2) - p_2 q_1 = -\dfrac{1}{2}
\end{aligned}
\right.
\quad\Leftrightarrow\quad
\left\{
\begin{aligned}
& p_1=q_1=-1\\
& p_2=q_2-\dfrac{1}{2}
\end{aligned}
\right..
\end{align*}

Next, we use condition \eqref{D3} to uniquely determine the pair $(p_2,q_2)$. Let $t=-L_f(x)$. Since $t\to\infty$ as $L_f(x)\to-\infty$, we obtain
\begin{align*}
\lim_{L_f(x)\rightarrow-\infty}\dfrac{R(L_f(x))}{1 - L_f(x)} = \lim_{t \rightarrow \infty} \dfrac{\sqrt{1 + t^2} + t}{1 + t} = 2.
\end{align*}
Since
\begin{align*}
\dfrac{B(L_f(x))}{1 - L_f(x)} = A(L_f(x)) \dfrac{R(L_f(x))}{1 - L_f(x)},
\end{align*}
condition \eqref{D3} requires
\begin{align*}
\lim_{L_f(x)\rightarrow-\infty} A(L_f(x)) = \dfrac{1}{2}.
\end{align*}
It follows from \eqref{A} and $p_2 = q_2 - 1/2$ that
\begin{align*}
\left\{
\begin{aligned}
& \dfrac{p_2}{q_2} = \dfrac{1}{2}\\
& p_2 = q_2 - \dfrac{1}{2}
\end{aligned}
\right. \quad \Leftrightarrow \quad 
\left\{
\begin{aligned}
& p_2 = \dfrac{1}{2}\\
& q_2 = 1
\end{aligned}
\right..
\end{align*}
Thus, within the rational form \eqref{A}, $A(L_f(x))$ is uniquely determined as
\begin{align*}
A(L_f(x)) = \dfrac{1 - L_f(x) + \frac{1}{2}L_f(x)^2}{1 - L_f(x) + L_f(x)^2}.
\end{align*}
Both the numerator and denominator are positive on $\mathbb{R}$, because
\begin{align*}
1 - L_f(x) + \dfrac{1}{2}L_f(x)^2 = \dfrac{1}{2}\left[(L_f(x)-1)^2 + 1\right] > 0
\end{align*}
and
\begin{align*}
1 - L_f(x) + L_f(x)^2 = \left(L_f(x)-\dfrac{1}{2}\right)^2 + \dfrac{3}{4} >0.
\end{align*}
Thus, $A(L_f(x)) > 0$ for all $L_f(x)\in\mathbb{R}$, and $B(L_f(x))$ is given by
\begin{align*}
B(L_f(x)) = A(L_f(x))R(L_f(x)) = \dfrac{1 - L_f(x) + \frac{1}{2}L_f(x)^2}{1 - L_f(x) + L_f(x)^2} \left(\sqrt{1 + L_f(x)^2} - L_f(x)\right).
\end{align*}
The proposed correction function is
\begin{align*}
C_*(L_f(x)) = \dfrac{2}{1 + \dfrac{1 - L_f(x) + \frac{1}{2}L_f(x)^2}{1 - L_f(x) + L_f(x)^2} (\sqrt{1 + L_f(x)^2} - L_f(x))}.
\end{align*}
Then the corresponding iterative method is defined by
\begin{align}\label{im}
x_{n+1} = x_n - C_*(L_f(x_n))\dfrac{f(x_n)}{f'(x_n)}.
\end{align}

\begin{rem}
For large positive values of $L_f(x)$, direct evaluation of $\sqrt{1+L_f(x)^2}-L_f(x)$ may suffer from loss of significance. To improve numerical stability, $R(L_f(x))$ can instead be evaluated as
\begin{align*}
R(L_f(x)) =
\begin{cases}
\dfrac{1}{\sqrt{1+L_f(x)^2}+L_f(x)},
& L_f(x)\geq 0\\[2mm]
\sqrt{1+L_f(x)^2}-L_f(x),
& L_f(x)<0
\end{cases}.
\end{align*}
These two expressions are mathematically equivalent.
\end{rem}

\section{Theoretical properties}\label{sec:3}
In this section, we establish fundamental properties of the correction function $C_*(L_f(x))$ proposed in Section \ref{sec:2} and investigate the local and global convergence of the corresponding iterative method.

First, Proposition~\ref{pr1} shows that $C_*(L_f(x))$ is well-defined, positive, and bounded on $\mathbb{R}$.
\begin{prop}\label{pr1}
The correction function $C_*(L_f(x))$ is well-defined on $\mathbb{R}$ and satisfies $0 < C_*(L_f(x)) < 2$ for all $L_f(x) \in \mathbb{R}$.
Moreover, 
\begin{align*}
\lim_{L_f(x) \rightarrow -\infty}C_*(L_f(x)) = 0, \quad \lim_{L_f(x) \rightarrow \infty}C_*(L_f(x)) = 2,
\end{align*}
and hence $C_*(\mathbb{R}) = (0,2)$.
\end{prop}
\begin{proof}
As shown in Section~\ref{sec:2}, $A(L_f(x))>0$ and $R(L_f(x))>0$ for all $L_f(x)\in\mathbb{R}$. 
Hence, $1 + B(L_f(x)) > 1$, so $C_*(L_f(x))$ is well-defined and positive on $\mathbb{R}$.
From
\begin{align*}
2 - C_*(L_f(x)) = \dfrac{2A(L_f(x))R(L_f(x))}{1 + A(L_f(x))R(L_f(x))} >0,
\end{align*}
we obtain $0 < C_*(L_f(x)) <2$ for all $L_f(x) \in \mathbb{R}$.

Next, since
\begin{align*}
\lim_{L_f(x) \rightarrow \pm \infty}A(L_f(x)) = \dfrac{1}{2}, \quad \lim_{L_f(x) \rightarrow \infty}R(L_f(x)) = 0, \quad \lim_{L_f(x) \rightarrow - \infty}R(L_f(x)) = \infty,
\end{align*} 
it follows that $\lim_{L_f(x) \rightarrow -\infty}C_*(L_f(x)) = 0$ and $\lim_{L_f(x) \rightarrow \infty}C_*(L_f(x)) = 2$.
Since $C_*(L_f(x))$ is continuous on $\mathbb{R}$, it follows that $C_*(\mathbb{R}) = (0,2)$.
This completes the proof.
\end{proof}

Next, we establish a pointwise comparison between the correction functions of the proposed method and Halley’s method.
\begin{prop}\label{pr_compare}
The following statements hold.\\
(i) If $L_f(x)<2$, then $0 < C_*(L_f(x)) \leq C_H(L_f(x))$.
Moreover, equality holds if and only if $L_f(x) = 0$.\\
(ii) If $L_f(x)\geq2$, then $0 < C_*(L_f(x)) < 2$.
In contrast, $C_H(L_f(x))$ is undefined at $L_f(x) = 2$ and satisfies $C_H(L_f(x))<0$ for $L_f(x) > 2$.
\end{prop}
\begin{proof}
From the definition of $R(L_f(x))$, $L_f(x)$ can be expressed as
\begin{align*}
&R(L_f(x)) + L_f(x) = \sqrt{1 + L_f(x)^2}\\
\Rightarrow \quad & R(L_f(x))^2 + 2R(L_f(x))L_f(x) + L_f(x)^2 = 1 + L_f(x)^2\\
\Leftrightarrow \quad &L_f(x) = \dfrac{1-R(L_f(x))^2}{2R(L_f(x))}.
\end{align*}
Using this identity, we obtain
\begin{align*}
&B(L_f(x))-(1-L_f(x))\\
&= \dfrac{\left(1 - L_f(x) + \frac{L_f(x)^2}{2}\right) R(L_f(x)) - (1- L_f(x))(1 - L_f(x) + L_f(x)^2)}{1 - L_f(x) + L_f(x)^2}\\
&=\dfrac{[R(L_f(x))-1]^4}
{8R(L_f(x))^3[L_f(x)^2-L_f(x)+1]}\\
&\geq 0.
\end{align*}
(i) Suppose that $L_f(x)<2$. 
Since 
\begin{align*}
1+B(L_f(x)) \geq 2-L_f(x) > 0,
\end{align*}
we have
\begin{align*}
C_*(L_f(x)) = \dfrac{2}{1+B(L_f(x))} \leq \dfrac{2}{2-L_f(x)} = C_H(L_f(x)).
\end{align*}
It follows from Proposition~\ref{pr1} that $C_*(L_f(x)) > 0$, and thus
$0 < C_*(L_f(x)) \leq C_H(L_f(x))$.
Moreover, $C_*(L_f(x)) = C_H(L_f(x))$ holds if and only if
\begin{align*}
&R(L_f(x)) = 1\\
\Leftrightarrow \quad & 1+L_f(x)^2 = (1 - L_f(x))^2\\
\Leftrightarrow \quad & L_f(x) = 0.
\end{align*}
(ii) Suppose that $L_f(x) \geq 2$.
By Proposition~\ref{pr1}, $0 < C_*(L_f(x)) < 2$ for all $L_f(x) \in \mathbb{R}$.
In contrast, $C_H(L_f(x))$ is undefined at $L_f(x)=2$, since the denominator in \eqref{halley} vanishes.
Moreover, $C_H(L_f(x))<0$ holds for $L_f(x)>2$.
This completes the proof.
\end{proof}

\begin{rem}
Proposition~\ref{pr_compare} shows that, for $L_f(x)<2$, the proposed method preserves the direction of the Newton step while using a correction factor no larger than Halley’s.
For $L_f(x)>2$, it avoids the singularity and sign reversal associated with Halley’s correction function.

To illustrate the effect of sign reversal, let $\alpha$ be a simple root of $f(x)$ and assume that $f'(x) \neq 0$ on an interval containing both $x_n$ and $\alpha$. Define
\begin{align*}
r(x_n)=\dfrac{f(x_n)}{f'(x_n)(x_n-  \alpha)}.
\end{align*}
By the mean value theorem, $r(x_n) > 0$. 
From \eqref{GE}, we have
\begin{align*}
x_{n+1} - \alpha = \left[1-C(L_f(x_n))r(x_n)\right](x_n-\alpha).
\end{align*}
If $C(L_f(x_n)) < 0$, then $|x_{n+1} - \alpha| > |x_n - \alpha|$.
Thus, a negative correction factor increases the distance from the root in the next iteration.
Such an increase due to sign reversal cannot occur for the proposed method, because $C_*(L_f(x))>0$.
\end{rem}

We next examine the local relationship between the proposed method and Halley’s method near $L_f(x)=0$.
\begin{prop}\label{pr2}
The condition \eqref{HM} holds as $L_f(x) \rightarrow 0$.
\end{prop}
\begin{proof}
Taylor series expansion of $B(L_f(x))$ about $L_f(x)=0$ yields
\begin{align*}
B(L_f(x)) &= \left(1 - \dfrac{1}{2}L_f(x)^2 - \dfrac{1}{2}L_f(x)^3 + O(L_f(x)^4)\right)\left(1- L_f(x) + \dfrac{1}{2}L_f(x)^2 + O(L_f(x)^4)\right)\\
& = 1 - L_f(x) + \dfrac{1}{2}L_f(x)^2 - \dfrac{1}{2}L_f(x)^2 + \dfrac{1}{2}L_f(x)^3 - \dfrac{1}{2}L_f(x)^3 + O(L_f(x)^4)\\
& = 1 - L_f(x) + O(L_f(x)^4).
\end{align*}
Hence, $C_*(L_f(x))$ has the expansion
\begin{align*}
C_*(L_f(x)) &= \dfrac{2}{1 + B(L_f(x))}\\
& = \dfrac{1}{1 - \frac{L_f(x)}{2} + O(L_f(x)^4)}\\
& = 1 + \dfrac{L_f(x)}{2} + \dfrac{L_f(x)^2}{4} + \dfrac{1}{8}L_f(x)^3 + O(L_f(x)^4).
\end{align*}
Similarly, Halley's correction function satisfies
\begin{align*}
C_H(L_f(x)) = \dfrac{2}{2-L_f(x)} = 1 + \dfrac{L_f(x)}{2} + \dfrac{L_f(x)^2}{4} + \dfrac{1}{8}L_f(x)^3 + O(L_f(x)^4).
\end{align*}
Therefore, $C_*(L_f(x)) = C_H(L_f(x)) + O(L_f(x)^4)$.
This completes the proof.
\end{proof}

\begin{rem}
Considering the Taylor expansions in Proposition~\ref{pr2} up to the fourth-order term, we obtain
\begin{align*}
C_*(L_f(x))
&= 1 + \dfrac{L_f(x)}{2}+\dfrac{L_f(x)^2}{4} + \dfrac{L_f(x)^3}{8} +O(L_f(x)^5),\\
C_H(L_f(x))
&= 1 + \dfrac{L_f(x)}{2}+\dfrac{L_f(x)^2}{4} +\dfrac{L_f(x)^3}{8}
+\dfrac{L_f(x)^4}{16}
+O(L_f(x)^5).
\end{align*}
In particular, the coefficient of $L_f(x)^4$ in $C_*(L_f(x))$ is zero. 
Then it holds that
\begin{align*}
C_*(L_f(x))-C_H(L_f(x)) = -\dfrac{1}{16}L_f(x)^4 + O(L_f(x)^5).
\end{align*}
Thus, the two correction functions have identical coefficients through the $L_f(x)^3$ term, and their difference first appears at order $L_f(x)^4$.
\end{rem}

Subsequently, we determine the local order of convergence of the proposed method.
\begin{thm}\label{thm1}
Let $\alpha$ be a simple root of $f(x)$, and suppose that $f(x)$ is sufficiently smooth around $\alpha$.
Suppose that the initial guess $x_0$ is sufficiently close to $\alpha$.
Then the method defined by \eqref{im} has at least third-order convergence and satisfies the error equation
\begin{align*}
e_{n+1} = (c_2^2 - c_3)e_n^3 + O(e_n^4),
\end{align*}
where $e_n=x_n-\alpha$ denotes the error at the $n$th iteration and
\begin{align*}
c_j = \dfrac{f^{(j)}(\alpha)}{j!f'(\alpha)},
\quad j=2,3.
\end{align*}
\end{thm}
\begin{proof}
By using Taylor’s expansion about $\alpha$, we have
\begin{align*}
&f(x_n) = f'(\alpha)\left[e_n + c_2e_n^2+ c_3e_n^3 + O(e_n^4)\right],\\
&f'(x_n) = f'(\alpha)\left[1 + 2c_2e_n+ 3c_3e_n^2 + O(e_n^3)\right],\\
&f''(x_n) = f'(\alpha)\left[2c_2 + 6c_3e_n + O(e_n^2)\right].
\end{align*}
From the above equations, we obtain
\begin{align*}
&f(x_n)f''(x_n) = f'(\alpha)^2 \left[2c_2e_n+(2c_2^2+6c_3)e_n^2+O(e_n^3) \right],\\
&f'(x_n)^2 = f'(\alpha)^2 \left[1 + 4c_2e_n + (4c_2^2+6c_3)e_n^2+O(e_n^3)
\right].
\end{align*}
Hence, $L_f(x_n)$ can be written in the following form:

\begin{align*}
L_f(x_n) &= \dfrac{f(x_n)f''(x_n)}{f'(x_n)^2}\\
&= \left[2c_2e_n+(2c_2^2+6c_3)e_n^2+O(e_n^3)
\right] \left[1-4c_2e_n+O(e_n^2)\right]\\
&= 2c_2e_n-6(c_2^2-c_3)e_n^2+O(e_n^3).
\end{align*}
Moreover, we have
\begin{align}\label{h1}
\dfrac{f(x_n)}{f'(x_n)} = e_n - c_2e_n^2 + 2(c_2^2-c_3)e_n^3 + O(e_n^4).
\end{align}
From the proof of Proposition~\ref{pr2}, the Taylor expansion of $C_*(L_f(x))$ about $L_f(x)=0$ is
\begin{align*}
C_*(L_f(x)) = 1 + \dfrac{1}{2}L_f(x) + \dfrac{1}{4}L_f(x)^2 + O(L_f(x)^3).
\end{align*}
Substituting $x=x_n$ into the above expansion gives
\begin{align}\label{h2}
C_*(L_f(x_n)) &= 1 + \dfrac{1}{2}L_f(x_n) + \dfrac{1}{4}L_f(x_n)^2 + O(L_f(x_n)^3) \notag\\
&= 1 + c_2e_n + (-2c_2^2 + 3c_3)e_n^2 + O(e_n^3).
\end{align}
By substituting \eqref{h1} and \eqref{h2} into \eqref{im}, we obtain
\begin{align*}
e_{n+1} &= e_n - C_*(L_f(x_n))\dfrac{f(x_n)}{f'(x_n)}\\
&= e_n - \left[1 + c_2e_n + (-2c_2^2 + 3c_3)e_n^2 + O(e_n^3)\right]\left[e_n - c_2e_n^2 + 2(c_2^2-c_3)e_n^3 + O(e_n^4)\right]\\
&= (c_2^2 -c_3)e_n^3 + O(e_n^4).
\end{align*}
This completes the proof.
\end{proof}

Finally, we investigate the global convergence of the proposed method.
We begin with a lemma that will be used to establish the monotonicity of the iteration function.
For the proof, see Appendix~\ref{app:lem1}.

\begin{lem}\label{lem1}
For all $L_f(x)\in\mathbb{R}$, it holds that
\begin{align*}
\dfrac{dC_*(L_f(x))}{dL_f(x)} > 0.
\end{align*}
Moreover, if $L_f(x) \geq -9/2,$
then $\Phi(x) \geq 0$, where
\begin{align*}
\Phi(x) = 1 - C_*(L_f(x))(1-L_f(x)) - L_f(x)\left(1-\dfrac{L_f(x)}{2}\right)\dfrac{dC_*(L_f(x))}{dL_f(x)}.
\end{align*}
\end{lem}

To simplify the notation, let us denote
\begin{align*}
L_{f'}(x) = \dfrac{f'(x)f^{(3)}(x)}{f''(x)^2}.
\end{align*}
From Lemma~\ref{lem1}, we obtain the following sufficient condition for the global convergence of the proposed method.

\begin{thm}\label{thm2}
Let $\alpha\in(a,b)$ be a simple root of $f$, where $f\in C^3([a,b])$.
Suppose that
\begin{align}\label{a1}
f'(x) \neq 0, \quad f''(x) \neq 0, \quad L_f(x) \geq -\dfrac{9}{2}, \quad L_{f'}(x) \leq \dfrac{3}{2}
\end{align}
for all $x\in[a,b]$.
Then, for any initial guess $x_0 \in [a,b]$ satisfying $f(x_0)f'(x_0) > 0$, the sequence $\{x_n\}$ defined by \eqref{im} decreases monotonically and converges to $\alpha$.
If $f(x_0)f'(x_0) < 0$, then $\{x_n\}$ increases monotonically and converges to $\alpha$.
\end{thm}
\begin{proof}
We define the iteration function of the proposed method by
\begin{align*}
F_*(x) = x - C_*(L_f(x))\dfrac{f(x)}{f'(x)}.
\end{align*}
From the definitions of $L_f(x)$ and $L_{f'}(x)$, we have
\begin{align*}
& \dfrac{f''(x)}{f'(x)}L_f(x)L_{f'}(x) = \dfrac{f''(x)}{f'(x)}\dfrac{f(x)f''(x)}{f'(x)^2}\dfrac{f'(x)f^{(3)}(x)}{f''(x)^2} = \dfrac{f(x)f^{(3)}(x)}{f'(x)^2},\\
& \dfrac{f''(x)}{f'(x)}L_f(x) =\dfrac{f''(x)}{f'(x)}\dfrac{f(x)f''(x)}{f'(x)^2} = \dfrac{f(x)f''(x)^2}{f'(x)^3}.
\end{align*}
Differentiating $L_f(x)$ with respect to $x$ gives
\begin{align*}
\dfrac{d L_f(x)}{dx} &= \dfrac{d}{dx}\dfrac{f(x)f''(x)}{f'(x)^2}\\
&= \dfrac{\left[f'(x)f''(x) + f(x)f^{(3)}(x)\right]f'(x)^2 - f(x)f''(x)\cdot 2f'(x)f''(x)}{f'(x)^4}\\
&= \dfrac{f''(x)}{f'(x)} + \dfrac{f(x)f^{(3)}(x)}{f'(x)^2} - 2 \dfrac{f(x)f''(x)^2}{f'(x)^3}\\
&= \dfrac{f''(x)}{f'(x)}
\left[1 + L_f(x)(L_{f'}(x) - 2)\right].
\end{align*}
Moreover, it holds that
\begin{align*}
\dfrac{d}{dx}\dfrac{f(x)}{f'(x)} = 1 - L_f(x).
\end{align*}
By differentiating $F_*(x)$ with respect to $x$, we obtain
\begin{align*}
\dfrac{dF_*(x)}{dx} &= 1 - \dfrac{dC_*(L_f(x))}{dx}\dfrac{f(x)}{f'(x)} - C_*(L_f(x))\dfrac{d}{dx}\dfrac{f(x)}{f'(x)}\\
&= 1 - \dfrac{dC_*(L_f(x))}{dL_f(x)}\dfrac{dL_f(x)}{dx}\dfrac{f(x)}{f'(x)} - C_*(L_f(x))(1 - L_f(x))\\
&= 1  - C_*(L_f(x))(1 - L_f(x)) - \dfrac{dC_*(L_f(x))}{dL_f(x)}L_f(x)\left[1 + L_f(x)(L_{f'}(x) - 2)\right]\\
&= 1 - C_*(L_f(x))(1 - L_f(x)) - \dfrac{dC_*(L_f(x))}{dL_f(x)}L_f(x)\left[\left(1 - \dfrac{L_f(x)}{2}\right) + L_f(x)\left(L_{f'}(x) - \dfrac{3}{2}\right) \right]\\
&= 1 - C_*(L_f(x))(1 - L_f(x)) - \dfrac{dC_*(L_f(x))}{dL_f(x)}L_f(x)\left(1 - \dfrac{L_f(x)}{2}\right)\\
&\quad - \dfrac{dC_*(L_f(x))}{dL_f(x)}L_f(x)^2\left(L_{f'}(x) - \dfrac{3}{2}\right)\\
&= \Phi(x) + \dfrac{dC_*(L_f(x))}{dL_f(x)}L_f(x)^2\left(\dfrac{3}{2} - L_{f'}(x)\right).
\end{align*}
From \eqref{a1} and Lemma~\ref{lem1},
\begin{align*}
\dfrac{dF_*(x)}{dx}\geq0
\end{align*}
for all $x\in[a,b]$. 
Hence, $F_*(x)$ is nondecreasing on $[a,b]$.
Since $f(\alpha)=0$, we also have $F_*(\alpha) = \alpha$.
Suppose first that $f(x_0)f'(x_0) > 0$.
Note that since $f'(x)\neq0$ for all $x\in[a,b]$, the sign of $f'(x)$ is constant on $[a,b]$. 
By the mean value theorem, there exists a $\xi \in (a,b)$ between $x_0$ and $\alpha$ such that
\begin{align}\label{mean}
\dfrac{f(x_0) - f(\alpha)}{x_0 - \alpha} = f'(\xi) \quad \Leftrightarrow \quad f(x_0) = f'(\xi)(x_0 - \alpha).
\end{align}
Multiplying both sides of \eqref{mean} by $f'(x_0) \neq 0$ gives
\begin{align*}
f(x_0)f'(x_0) = f'(\xi)f'(x_0)(x_0 - \alpha).
\end{align*}
It follows from $f'(\xi)f'(x_0) >0$ and $f(x_0)f'(x_0) > 0$ that $x_0 > \alpha$.
Because $F_*(x)$ is nondecreasing,
\begin{align*}
x_1 = F_*(x_0) \geq F_*(\alpha) = \alpha.
\end{align*}
Additionally, we have
\begin{align*}
x_1 = x_0 - C_*(L_f(x_0)) \dfrac{f(x_0)}{f'(x_0)} < x_0
\end{align*}
from $C_*(L_f(x_0))>0$ and $f(x_0)/f'(x_0) > 0$.
Therefore, $\alpha \leq x_1 < x_0$ is satisfied.
By induction, it is easily shown that $\alpha\leq x_{n+1}<x_n$ for all $n \in \mathbb{N}$.
Hence, ${x_n}$ is monotonically decreasing and bounded below, and therefore converges to some $r \geq \alpha$.

Taking the limit as $n\to\infty$ in \eqref{im}, we obtain
\begin{align*}
r = r - C_*(L_f(r))\dfrac{f(r)}{f'(r)} \quad \Leftrightarrow \quad C_*(L_f(r))\dfrac{f(r)}{f'(r)} = 0.
\end{align*}
From $C_*(L_f(r))>0$ and $f'(r) \neq 0$, we have $f(r) = 0$.
Since $f'(x)\neq0$ on $[a,b]$, the function $f(x)$ is strictly monotone on this interval. 
This implies that $\alpha$ is the unique root of $f(x)$ in $[a,b]$, and hence $r = \alpha$.

If $f(x_0)f'(x_0)<0$, the same argument gives $x_0 < \alpha$ and $x_n<x_{n+1} \leq \alpha$ for all $n \in \mathbb{N}$.
Thus, ${x_n}$ is increasing and bounded above. Taking the limit as above shows that it converges to $\alpha$.
This completes the proof.
\end{proof}

\begin{rem}
A global convergence result for Halley's method given in \cite{RefB20} is established under the conditions
\begin{align*}
f'(x) \neq 0, \quad f''(x) \neq 0, \quad L_f(x) < \dfrac{1}{2}, \quad L_{f'}(x) \leq \dfrac{3}{2}
\end{align*}
for all $x\in[a,b]$.
On the other hand, Theorem~\ref{thm2} guarantees the monotonic convergence of the proposed method under \eqref{a1}.
Thus, the two results impose the same conditions on $f'(x)$, $f''(x)$, and $L_{f'}(x)$, while the condition on $L_f(x)$ is different.
In particular, the sufficient condition for the proposed method imposes no upper bound on $L_f(x)$, whereas the corresponding result for Halley's method requires $L_f(x) < 1/2$.
\end{rem}

\section{Numerical results}\label{sec:4}
In this section, we examine the effectiveness of the proposed method through several numerical examples.
All numerical computations are performed in double precision using MATLAB R2026a.
\subsection{Numerical experiments}\label{sec:41}
In this subsection, we apply the Halley, Chebyshev, super-Halley, Fang, Ostrowski, Noor, Extended Newton, Padé-type, and proposed methods to the following test functions and compare their convergence performance over a wide range of initial guesses:
\begin{itemize}
    \item $f_1(x) = (x-1)^2-\log(x+1), \quad
    \alpha_1 \approx 0.412391172023885, \ 
    \alpha_2 \approx 2.057103549994738$.
    
    \item $f_2(x) = \cos x-x, \quad
    \alpha \approx 0.739085133215161$.
    
    \item $f_3(x) = x^3 + 4x^2 - 10, \quad
    \alpha \approx 1.365230013414097$.
    
    \item $f_4(x) = (x-1)^3-1, \quad
    \alpha = 2$.
    
    \item $f_5(x) = e^{-x^2+x+2}-1, \quad
    \alpha_1 = -1, \ 
    \alpha_2 = 2$.
    
    \item $f_6(x) = \sin^2 x-x^2+1, \quad
    \alpha_1 \approx -1.404491648215341, \ 
    \alpha_2 \approx 1.404491648215341$.
    
    \item $f_7(x) = (x-3)e^x+1, \quad
    \alpha_1 \approx -1.505241495792883, \ 
    \alpha_2 \approx 2.947530902542285$.
\end{itemize}
The range of initial guesses is set to $x_0\in[0,5]$ for $f_1$ and $x_0\in[-5,5]$ for $f_2,\ldots,f_7$. 
Each interval is partitioned into 100 equal subintervals, and the resulting 101 points are used as initial guesses for each method.
For each test function, convergence is declared when
\begin{align*}
|f_i(x_{n+1})| \leq 10^{-12}, \quad i = 1,\ldots,7.
\end{align*}
The maximum number of iterations is set to $N_{\max}=100$. 
Convergence is considered unsuccessful if the stopping criterion is not satisfied within 100 iterations, if the iteration becomes undefined over the real numbers, or if \texttt{NaN} or \texttt{Inf} occurs. 
For $f_1$, $f_5$, $f_6$, and $f_7$, which have multiple roots, convergence to any root is regarded as successful.

To evaluate the robustness of each method with respect to the choice of initial guess, we use the convergence success rate (CSR) and the average number of iterations (ANI).
Let $M$ denote the total number of initial guesses and $M_{\mathrm{c}}$ the number of initial guesses for which convergence is successful. 
CSR is defined by
\begin{align*}
\mathrm{CSR} = \dfrac{M_{\mathrm{c}}}{M} = \dfrac{M_{\mathrm{c}}}{101}.
\end{align*}
A higher CSR indicates that the method converges successfully for a broader range of initial guesses.
Let $N_j$ denote the number of iterations required for convergence from an initial guess $x_0^{(j)}$ for which convergence is achieved.
For each initial guess for which convergence is not achieved, $N_{\max}$ is assigned as the iteration count.
Then ANI is defined as
\begin{align*}
\mathrm{ANI}
= \dfrac{1}{M}\left[\sum_{j\in\mathcal{S}}N_j + (M -M_{\mathrm{c}})N_{\max}\right],
\end{align*}
where $\mathcal{S}$ denotes the index set of initial guesses for which convergence is successful.
Since ANI assigns the maximum iteration count $N_{\max}$ as a penalty to initial guesses for which convergence is not achieved, it reflects not only the iterative efficiency of successful convergence but also the effect of convergence failures.
CSR and ANI obtained for each method are presented in Table~\ref{table2}.
CSR is reported as a percentage.

\begin{table}[t]
\centering
\caption{Numerical results for different methods applied to the test functions $f_1$ to $f_7$.}
\label{table2}
\resizebox{\textwidth}{!}{
\begin{tabular}{c|l|ccccccc|c}
\hline
 & Method
 & $f_1(x)$
 & $f_2(x)$
 & $f_3(x)$
 & $f_4(x)$
 & $f_5(x)$
 & $f_6(x)$
 & $f_7(x)$
 & Average\\
\hline

\multirow{9}{*}{CSR (\%)}
& Halley
& \textbf{100.00} & 94.06 & 95.05 & \textbf{99.01}
& \textbf{99.01} & \textbf{99.01} & \textbf{99.01}
& 97.88\\

& Chebyshev
& 94.06 & 36.63 & 98.02 & \textbf{99.01}
& 31.68 & \textbf{99.01} & 72.28
& 75.81\\

& super-Halley
& 98.02 & 87.13 & 98.02 & \textbf{99.01}
& 45.54 & \textbf{99.01} & 98.02
& 89.25\\

& Fang
& \textbf{100.00} & 59.41 & 26.73 & 13.86
& 57.43 & 73.27 & 82.18
& 58.98\\

& Ostrowski
& \textbf{100.00} & 65.35 & 49.50 & 50.50
& 59.41 & \textbf{99.01} & \textbf{99.01}
& 74.68\\

& Noor
& 94.06 & 34.65 & \textbf{99.01} & 98.02
& 29.70 & \textbf{99.01} & 71.29
& 75.11\\

& Extended Newton
& \textbf{100.00} & 73.27 & 81.19 & 90.10
& \textbf{99.01} & \textbf{99.01} & \textbf{99.01}
& 91.65\\

& Pad\'e-type
& \textbf{100.00} & 73.27 & \textbf{99.01} & \textbf{99.01}
& \textbf{99.01} & \textbf{99.01} & \textbf{99.01}
& 95.47\\

& Proposed
& \textbf{100.00} & \textbf{100.00} & 96.04 & \textbf{99.01}
& \textbf{99.01} & \textbf{99.01} & \textbf{99.01}
& \textbf{98.87}\\

\hline

\multirow{9}{*}{ANI}
& Halley
& 3.594 & 12.733 & 23.931 & 8.178
& \textbf{7.673} & 5.188 & 4.663
& 9.423\\

& Chebyshev
& 9.723 & 64.733 & \textbf{15.970} & 11.713
& 69.485 & 5.901 & 31.079
& 29.801\\

& super-Halley
& 4.960 & 21.089 & 17.030 & \textbf{6.812}
& 56.099 & 8.000 & 5.733
& 17.103\\

& Fang
& \textbf{2.772} & 42.683 & 74.099 & 86.554
& 44.733 & 29.109 & 20.861
& 42.973\\

& Ostrowski
& 3.347 & 36.980 & 52.178 & 51.465
& 42.772 & \textbf{4.931} & \textbf{4.000}
& 27.953\\

& Noor
& 9.416 & 66.495 & 16.020 & 14.089
& 71.228 & 12.317 & 31.693
& 31.608\\

& Extended Newton
& 3.743 & 30.535 & 29.386 & 20.010
& 10.762 & 5.327 & 5.010
& 14.968\\

& Pad\'e-type
& 3.624 & 30.050 & 16.861 & 10.267
& 7.713 & 4.990 & 4.762
& 11.181\\

& Proposed
& 3.594 & \textbf{7.683} & 19.594 & 7.465
& 7.693 & 4.990 & 4.713
& \textbf{7.962}\\

\hline
\end{tabular}
}
\end{table}

Table~\ref{table2} shows that the proposed method achieves the highest average CSR and the lowest average ANI over the seven test functions. However, the proposed method does not necessarily attain the smallest ANI for each test function.
For example, Fang's method gives the best ANI for $f_1$, super-Halley's method for $f_4$, and Ostrowski's method for $f_6$ and $f_7$. 
But it should be noted that these methods exhibit substantially lower CSR values for some of the test functions.
In contrast, the proposed method achieves a CSR of at least $96\%$ for all test functions and attains the best values for both CSR and ANI for $f_2$.
This indicates that the proposed method provides a favorable balance between robustness with respect to the initial guess and iterative efficiency across the test functions.

The preceding numerical experiments confirm that both the proposed method and Halley’s method exhibit high convergence performance.
To further clarify the structural difference between the two methods, we examine the behavior near $L_f(x)=2$, where Halley’s correction function becomes singular.
Consider the test function
\begin{align*}
f_8(x) = \log x, \quad \alpha = 1,
\end{align*}
and apply Halley’s method and the proposed method from initial guesses $x_0$ chosen so that
\begin{align*}
L_f(x_0) = 1.9, 1.99, 1.9999, 1.99999999, 2.0, 2.00000001, 2.0001, 2.01, 2.1.
\end{align*}
The stopping criterion and the maximum number of iterations are set to $|f_8(x_{n+1})|\leq10^{-12}$ and $N_{\max}=100$, respectively.
Table~\ref{table3} shows the correction factor $C(L_f(x_0))$, the first iterate $x_1$, and the number of iterations $N$ required for convergence at each initial guess.
\begin{table}[t]
\centering
\caption{Numerical results for two methods applied to the test function $f_8$.}
\label{table3}
\begin{tabular}{c|ccc|ccc}
\hline
$L_f(x_0)$
& \multicolumn{3}{c|}{Halley}
& \multicolumn{3}{c}{Proposed}\\
& $C_H(L_f(x_0))$
& $x_1$
& $N$
& $C_*(L_f(x_0))$
& $x_1$
& $N$\\
\hline

$1.9$
& $2.0\times10^{1}$
& $5.833176$
& $5$
& $1.847548$
& $0.674606$
& $4$\\

$1.99$
& $2.0\times10^{2}$
& $54.541475$
& --
& $1.853492$
& $0.640890$
& $4$\\

$1.9999$
& $2.0\times10^{4}$
& $5413.817349$
& --
& $1.854096$
& $0.637223$
& $4$\\

$1.99999999$
& $2.0\times10^{8}$
& $5.41341164\times10^{7}$
& --
& $1.854102$
& $0.637186$
& $4$\\

$2.0$
& --
& --
& --
& $1.854102$
& $0.637186$
& $4$\\

$2.00000001$
& $-2.0\times10^{8}$
& $-5.41341132\times10^{7}$
& --
& $1.854102$
& $0.637186$
& $4$\\

$2.0001$
& $-2.0\times10^{4}$
& $-5413.005337$
& --
& $1.854108$
& $0.637149$
& $4$\\

$2.01$
& $-2.0\times10^{2}$
& $-53.729459$
& --
& $1.854702$
& $0.633492$
& $4$\\

$2.1$
& $-2.0\times10^{1}$
& $-5.020714$
& --
& $1.859726$
& $0.600701$
& $4$\\

\hline
\end{tabular}
\end{table}

Table~\ref{table3} shows that, for Halley’s method, both $|C_H(L_f(x_0))|$ and $|x_1|$ increase rapidly as $L_f(x_0) \rightarrow 2$, while the correction factor becomes negative for $L_f(x_0) > 2$ and is undefined at $L_f(x_0)=2$. 
In contrast, for the proposed method, $C_*(L_f(x_0)) \approx 1.854$ and $x_1$ remain finite near $L_f(x_0)=2$, and convergence is achieved in four iterations for all initial guesses. 
These results indicate that the proposed method avoids both the singularity and sign reversal of Halley’s correction function and exhibits stable iterative behavior near $L_f(x)=2$.

\subsection{Application}\label{sec:42}
In this subsection, we consider the following van der Waals equation \citep{RefAP10} to examine the applicability of the proposed method to a practical problem and its convergence performance with respect to the initial guess:
\begin{align}\label{fe}
\left(P + \dfrac{a}{V_m^2}\right)(V_m - b) = RT,
\end{align}
where $P$ is the pressure, $V_m$ is the molar volume, $R$ is the universal gas constant, $T$ is the absolute temperature, and $a$ and $b$ are specific parameters for each gas that represent molecular attractions and molecular repulsions, respectively.
Equation~\eqref{fe} can be rewritten as the following nonlinear equation with respect to $V_m$:
\begin{align*}
f(V_m) = V_m^3 - \left(b + \dfrac{RT}{P}\right)V_m^2 + \dfrac{a}{P}V_m - \dfrac{ab}{P} = 0.
\end{align*}
Following the parameter setting for carbon dioxide used in \cite{RefZ24}, we set
\begin{align*}
P = 100 \ \mathrm{atm}, \quad  R = 0.082057, \quad T = 500 \ \mathrm{K}, \quad a = 3.6073, \quad b = 0.042816.
\end{align*}
In this case, the nonlinear equation becomes
\begin{align*}
f_9(x) = x^3 - 0.45310x^2 + 0.036073x - 1.5445 \times 10^{-3} = 0.
\end{align*}
This equation has a unique physically admissible root, given by
\begin{align*}
\alpha \approx 0.3660879312575294.
\end{align*}

In this experiment, we use three initial guesses, $x_0=0.25$, $0.5$, and $1.0$.
The stopping criterion and the maximum number of iterations are set to the same values as those used in Section~\ref{sec:41}.
Since $x$ represents the molar volume, we restrict the iteration to the physically admissible region $x > b = 0.042816$.
Convergence is regarded as unsuccessful if $x_n\leq b$ occurs during the iteration, if the correction function is not defined in the real domain, or if the stopping criterion is not satisfied within the maximum number of iterations. Table~\ref{table4} presents the number of iterations $N$ required to satisfy the stopping criterion when the nine methods are applied to each initial guess.

\begin{table}[t]
\centering
\caption{Numerical results for different methods applied to the test function $f_9$.}
\label{table4}
\begin{tabular}{l|ccc}
\hline
Method
& $x_0=0.25$
& $x_0=0.5$
& $x_0=1.0$
\\
\hline

Halley
& -- & 4 & 5
\\

Chebyshev
& 17 & 4 & 6
\\

super-Halley
& -- & \textbf{3} & \textbf{4}
\\

Fang
& -- & -- & --
\\

Ostrowski
& -- & \textbf{3} & \textbf{4}
\\

Noor
& 27 & \textbf{3} & \textbf{4}
\\

Extended Newton
& -- & 4 & 5
\\

Pad\'e-type
& 20 & 4 & 5
\\

Proposed
& \textbf{9} & 4 & 5
\\

\hline
\end{tabular}
\end{table}

Table~\ref{table4} shows that, for $x_0=0.5$ and $1.0$, the proposed method satisfies the stopping criterion in the same number of iterations as Halley’s method. 
The super-Halley, Ostrowski, and Noor methods converge in one fewer iteration than the proposed method; thus, although the proposed method does not always achieve convergence in the fewest iterations, it exhibits comparable convergence performance in terms of iteration count. 
On the other hand, for $x_0=0.25$, the proposed method converges even though $|L_f(x_0)|\approx349.7462$ is extremely large, suggesting stable convergence behavior in this case.

\section{Concluding remarks}\label{sec:5}
In this paper, we proposed a new one-point iterative method based on a correction function that is positive and bounded on $\mathbb{R}$ while preserving the local behavior of Halley’s method.
A key feature of the proposed method is that its correction function matches that of Halley’s method to high order near the root and avoids both the singularity at $L_f(x) = 2$ and the sign reversal for $L_f(x) > 2$.
Moreover, we proved the third-order convergence of the proposed method and its global convergence under certain conditions.

In the numerical experiments, the proposed method achieved high convergence success rates and good iterative efficiency over a wide range of initial guesses and exhibited stable convergence behavior even near the singularity of Halley’s correction function. 
Furthermore, in the application to the van der Waals equation, the proposed method maintained convergence performance comparable to Halley’s method and also showed stable convergence behavior for an initial guess yielding a large value of $|L_f(x_0)|$.

\section*{Appendix}
\appendix
\section{Proof of Lemma~\ref{lem1}}\label{app:lem1}
\begin{proof}
Since $B(L_f(x))>0$, logarithmic differentiation of $B(L_f(x))$ with respect to $L_f(x)$ gives
\begin{align*}
\dfrac{1}{B(L_f(x))} \dfrac{dB(L_f(x))}{dL_f(x)}
&= \dfrac{2(L_f(x) - 1)}{L_f(x)^2 - 2L_f(x)+2}-\dfrac{2L_f(x) - 1}{L_f(x)^2 -L_f(x) + 1}
- \dfrac{1}{\sqrt{1 + L_f(x)^2}}\\
&= \dfrac{L_f(x)(L_f(x) - 2)}{(L_f(x)^2 - 2L_f(x)+2) (L_f(x)^2 - L_f(x)+1)} -\dfrac{1}{\sqrt{1 + L_f(x)^2}}.
\end{align*}
Suppose first that $0\leq L_f(x)\leq2$. 
Since $L_f(x)(L_f(x)-2) \leq 0$,
it holds that
\begin{align}\label{eq1}
\dfrac{dB(L_f(x))}{dL_f(x)} < 0.
\end{align}
Next, suppose that $L_f(x) > 2$. 
In this case, it follows from 
\begin{align*}
& L_f(x)^2 - 2L_f(x) + 2 = L_f(x)(L_f(x)-2) + 2 > L_f(x)(L_f(x) - 2),\\
& L_f(x)^2 - L_f(x) + 1 > L_f(x) + 1 > \sqrt{1 + L_f(x)^2}
\end{align*}
that
\begin{align*}
(L_f(x)^2-2L_f(x)+2)(L_f(x)^2-L_f(x)+1) > L_f(x)(L_f(x)-2)\sqrt{1+L_f(x)^2}.
\end{align*}
Hence, we have \eqref{eq1}.
Finally, suppose that $L_f(x)<0$. 
From
\begin{align*}
& L_f(x)^2 - 2L_f(x) + 2 = L_f(x)(L_f(x)-2) + 2 > L_f(x)(L_f(x)-2),\\
& L_f(x)^2 - L_f(x) + 1 > 1 - L_f(x) > \sqrt{1 + L_f(x)^2},
\end{align*}
we obtain
\begin{align*}
(L_f(x)^2-2L_f(x)+2)(L_f(x)^2-L_f(x)+1) > L_f(x)(L_f(x)-2)\sqrt{1+L_f(x)^2}.
\end{align*}
Hence, the same argument shows that \eqref{eq1} is satisfied for $L_f(x)<0$. 
Therefore, \eqref{eq1} holds for all $L_f(x)\in\mathbb{R}$.

On the other hand, differentiating
\begin{align*}
C_*(L_f(x)) = \dfrac{2}{1+B(L_f(x))}
\end{align*}
with respect to $L_f(x)$ gives
\begin{align*}
\dfrac{dC_*(L_f(x))}{dL_f(x)} = -\dfrac{2}{[1+B(L_f(x))]^2}\dfrac{dB(L_f(x))}{dL_f(x)}.
\end{align*}
It follows from \eqref{eq1} that
\begin{align*}
\dfrac{dC_*(L_f(x))}{dL_f(x)}>0
\end{align*}
for all $L_f(x)\in\mathbb{R}$.

Next, we examine the sign of $\Phi(x)$.
Substituting $C_*(L_f(x))$ and $\frac{dC_*(L_f(x))}{dL_f(x)}$ into $\Phi(x)$, we obtain
\begin{align*}
\Phi(x) = \dfrac{L_f(x)^3\left[\{4L_f(x)^2-9L_f(x)+8\}\sqrt{1+L_f(x)^2}+4L_f(x)^3-8L_f(x)^2+14L_f(x)-8\right]}{\sqrt{1+L_f(x)^2}\left[2\{L_f(x)^2-L_f(x)+1\}+\{L_f(x)^2-2L_f(x)+2\}\{\sqrt{1+L_f(x)^2}-L_f(x)\}\right]^2}.
\end{align*}
Since the denominator is positive for all $L_f(x)\in\mathbb{R}$, it suffices to determine the sign of the numerator.
Suppose first that $L_f(x)>0$.
Since
\begin{align*}
4L_f(x)^2-9L_f(x)+8
=
4\left(L_f(x)-\dfrac98\right)^2+\dfrac{47}{16}
>0
\end{align*}
and  $\sqrt{1+L_f(x)^2} \geq 1$, part of the numerator can be evaluated as
\begin{align*}
&\{4L_f(x)^2 - 9L_f(x) + 8\}\sqrt{1 + L_f(x)^2} + 4L_f(x)^3 - 8L_f(x)^2 + 14L_f(x) - 8\\
&\geq 4L_f(x)^3-4L_f(x)^2 + 5L_f(x)\\
&= L_f(x)\{4L_f(x)^2 - 4L_f(x) + 5\}\\
&= L_f(x) \left\{4\left(L_f(x) - \dfrac{1}{2}\right)^2 + 4\right\}\\
&>0.
\end{align*}
It follows from $L_f(x)^3>0$ that $\Phi(x)>0$ is satisfied.
Next suppose that $-9/2 \leq L_f(x) < 0$ and let $t=-L_f(x)$.
Then $0 < t \leq 9/2$, and we have
\begin{align*}
&\{4L_f(x)^2-9L_f(x)+8\}\sqrt{1+L_f(x)^2}
+4L_f(x)^3-8L_f(x)^2+14L_f(x)-8\\
&=(4t^2+9t+8)\sqrt{1+t^2} - (4t^3+8t^2+14t+8)\\
& = (4t^2+9t+8)\left(t+\dfrac{1}{\sqrt{1+t^2}+t}\right) -(4t^3+8t^2+14t+8)\\
& = t^2-6t-8 + \dfrac{4t^2+9t+8}{\sqrt{1+t^2}+t}.
\end{align*}
For $0 < t \leq 1$, we obtain $\sqrt{1 + t^2}>1$, and hence
\begin{align*}
t^2-6t-8 + \dfrac{4t^2+9t+8}{\sqrt{1+t^2}+t} &< t^2-6t-8 + \dfrac{4t^2+9t+8}{1+t}\\
&= \dfrac{t(t^2-t-5)}{1+t}\\
&< 0.
\end{align*}
For $1 \leq t \leq 9/2$, since $\sqrt{1 + t^2} > t$, the inequality holds:
\begin{align*}
t^2-6t-8 + \dfrac{4t^2 + 9t + 8}{\sqrt{1 + t^2} + t}
&< t^2 - 6t - 8 + \dfrac{4t^2 + 9t + 8}{2t}\\
&= t^2 - 4t - \dfrac{7}{2}+\dfrac{4}{t}.
\end{align*}
To simplify the notation, let us denote
\begin{align*}
h(t)=t^2 - 4t - \dfrac{7}{2} + \dfrac{4}{t}.
\end{align*}
Differentiating $h(t)$ twice with respect to $t$ gives
\begin{align*}
h''(t)= 2 + \dfrac{8}{t^3}>0.
\end{align*}
Since $h$ is convex on $[1,9/2]$, we have
\begin{align*}
h(t) \leq
\max\left\{h(1), h\left(\dfrac{9}{2}\right)\right\} 
= \max\left\{-\dfrac{5}{2}, -\dfrac{13}{36}\right\} < 0.
\end{align*}
Therefore,
\begin{align*}
(4t^2 + 9t + 8)\sqrt{1 + t^2} -(4t^3 + 8t^2 + 14t + 8) < 0
\end{align*}
is satisfied  for all $0 < t\leq 9/2$.
Since $L_f(x)^3<0$, we obtain $\Phi(x)>0$.

Finally, when $L_f(x)=0$, it is immediate that $\Phi(x)=0$.
This completes the proof.
\end{proof}
\end{document}